\documentclass[a4poaper,12pt,reqno, english]{article}
\usepackage{amssymb, amsthm, amsmath, amsfonts}
\usepackage[left=1in,right=1in,top=1in,bottom=1in]{geometry}
\usepackage[colorlinks=true,urlcolor=blue,linkcolor=red,citecolor=magenta]{hyperref}
\usepackage{cleveref}
\usepackage{graphicx}

\usepackage[backend=biber,maxbibnames=9, hyperref,
            style=numeric]{biblatex}
\AtBeginBibliography{\small}

\providecommand{\keywords}[1]{\textbf{Keywords} #1 \\[12px]}
\providecommand{\subclass}[1]{\textbf{Mathematics Subject Classification (2020)} #1 }

\newtheorem{theorem}{Theorem}[section]
\newtheorem{lemma}[theorem]{Lemma}
\newtheorem{corollary}[theorem]{Corollary}
\newtheorem{proposition}[theorem]{Proposition}
\newtheorem{conjecture}[theorem]{Conjecture}

\theoremstyle{definition}
\newtheorem{definition}[theorem]{Definition}
\newtheorem{remark}[theorem]{Remark}

\newtheorem{question}[theorem]{Question}

\DeclareMathOperator{\wt}{wt} 
\DeclareMathOperator{\im}{Im}
\DeclareMathOperator{\mult}{mult}
\DeclareMathOperator{\tr}{tr}

\DeclareMathOperator{\Cay}{Cay}
\DeclareMathOperator{\aff}{aff}
\DeclareMathOperator{\excludes}{\mathbf{X}}

\newcommand{\graph}[1]{\mathcal{G}_{#1}}

\newcommand{\N}{\mathbb{N}}
\newcommand{\Z}{\mathbb{Z}}

\newcommand{\F}{\mathbb{F}}
\newcommand{\exc}[1]{\excludes(#1)}
\newcommand{\set}[1]{\left\{ #1 \right\}}

\newcommand{\parens}[1]{ \left( #1 \right)}

\begin{document}

\title{On generalizing cryptographic results to Sidon sets in $\mathbb{F}_2^n$}

% % Article
\author{Darrion Thornburgh$^1$}
\date{$^1$Vanderbilt University \\ \texttt{darrion.thornburgh@vanderbilt.edu}\\[2ex]%
}

\maketitle
\begin{abstract}
    A Sidon set $S$ in $\F_2^n$ is a set such that $x+y=z+w$ has no solutions $x,y,z,w \in S$ with $x,y,z,w$ all distinct.
    In this paper, we prove various results on Sidon sets by using or generalizing known cryptographic results.
    In particular, we generalize known results on the Walsh transform of almost perfect nonlinear (APN) functions to Sidon sets.
    One such result is that we classify Sidon sets with minimal linearity    as those that are $k$-covers.
    That is, Sidon sets with minimal linearity are those Sidon sets $S \subseteq \F_2^n$ such that there exists $k > 0$ such that for any $p \in \F_2^n \setminus S$, there are exactly $k$ subsets $\set{x,y,z} \subseteq S$ such that $x+y+z = p$.
    From this, we also classify $k$-covers by means of the Cayley graph of a particular Boolean function, and we construct the unique rank $3$ strongly regular graph with parameters $(2048, 276, 44, 36)$ as the Cayley graph of a Boolean function.
    Finally, by computing the linearity of a particular family of Sidon sets, we increase the best-known lower bound of the largest Sidon set in  $\F_2^{4t+1}$ by $1$ for all $t \geq 4$.
\end{abstract}


\subclass{Primary: 11B13, 94D10.\\ Secondary: 05E30.}

\section{Introduction}

\begin{definition}
    A set $S$ in an abelian group $G$ is called a Sidon set if $a+b =c+d$ where $a,b,c,d \in S$ implies $\set{a,b} = \set{c,d}$.
    When $G$ is an elementary abelian $2$-group, we require that $a \neq b$ and $c \neq d$.
\end{definition}

In this paper, we will study Sidon sets in $G = \F_2^n$.
Sidon sets in $\F_2^n$ have various connections to symmetric cryptography.
In particular, a function $F \colon \F_2^n \to \F_2^n$ is called almost perfect nonlinear (APN) if and only if its graph 
$
{\graph F = \set{(x,F(x)) : x \in \F_2^n}}
$
is Sidon set.
APN functions are studied in symmetric cryptography for their utility in a block cipher and optimal resistance to a differential attack \cite{NybergBook1994}.
There also exists another type of attack on block ciphers called the linear attack, and the class of functions that are optimally resistant are called almost bent (AB) \cite{MatsuiLinearAttack} \cite{chabaud_vaudenay_1995}, and it turns out that all AB functions are also APN.
In this paper, we generalize known results on both APN and AB functions to Sidon sets and $k$-covers.

\begin{definition}
    A Sidon set $S \subseteq \F_2^n$ is called a \textbf{$k$-cover} if there exists $k > 0$ such that for any $p \in \F_2^n \setminus S$, there are exactly $k$ distinct sets $\set{x,y,z} \subseteq S$ such that $p = x+y+z$.
\end{definition}
By definition, $k$-covers are Sidon sets, but note that the notion of a $k$-cover could be generalized to include sets that are not Sidon sets.
A well-known classification of AB functions is that a function $F \colon \F_2^n \to \F_2^n$ is AB if and only if $\graph F$ is a $\parens{\frac{2^n-2}{6}}$-cover \cite{vandamflass}.
AB functions are also exactly those functions $F \colon \F_2^n \to \F_2^n$ such that the function
\[
\gamma_F(a,b) = \begin{cases}
    1 & \text{if } a \neq 0 \text{ and } F(x+a)+F(x)=b \text{ has a solution}, \\
    0 & \text{otherwise},
\end{cases}
\]
is bent, that is, the map $(a,b) \mapsto \sum_{(x,y) \in (\F_2^n)}(-1)^{\gamma_F(x,y) + a \cdot x + b \cdot y}$ always takes value in $\set{\pm 2^n}$ if and only if $F$ is AB \cite{carletCharpinZinovievCodesBentDES}.
Note that $\gamma_F(a,b) = 1$ if and only if $(a,b) \neq (0,0)$ and $((a,b) + \graph F) \cap \graph F \neq \emptyset$.
Generalizing this, we define for a set $S \subseteq \F_2^n$ the function $\gamma_S \colon \F_2^n \to \F_2$ such that
\[
\gamma_S(a) = \begin{cases}
    1 & \text{if } a \neq 0 \text{ and } (a + S) \cap S \neq \emptyset \\
    0 & \text{otherwise}.
\end{cases}
\]
Indeed, the equality $\gamma_F = \gamma_{\graph F}$ holds for any function $F \colon \F_2^n \to \F_2^n$.
One of our primary areas of focus in this paper is studying the Cayley graph of $\gamma_S$.

\begin{definition}
    The Cayley graph $\Cay(f)=(V,E)$ of a Boolean function $f \colon \F_2^n \to \F_2$ is the graph whose vertex set is $V =\F_2^n$ and edge set is $E=\set{\set{u,v} : f(u+v) = 1}$.
\end{definition}
\begin{definition}
    A graph $\Gamma$ is called strongly regular with parameters $(v,k, \lambda, \mu)$ if $\Gamma$ is a $k$-regular graph with $v$ vertices such that any two adjacent (non-adjacent resp.) vertices have $\lambda$ ($\mu$ resp.) neighbors in common.
\end{definition}

In two papers, \cite{BernasconiCodenotti} and \cite{BernasconiCodenottiVanderkam}, it was shown that the Cayley graph of a Boolean function $f \colon \F_2^n \to \F_2$ is strongly regular with $\lambda = \mu$ if and only if $f$ is bent.
A consequence of this result is that for any $F \colon \F_2^n \to \F_2^n$, the Cayley graph of $\gamma_F$ is strongly regular with $\lambda = \mu$ if and only if $F$ is AB.
Said differently, $\gamma_F$ is strongly regular with $\lambda = \mu$ if and only if $\graph F$ is a $k$-cover.

In this paper, we show that all $k$-covers can be classified in a similar manner, including those that are not graphs of AB functions, with an additional condition that we call ``separability''.
A set $S \subseteq \F_2^n$ is called separable if exactly half of its points lie in an affine hyperplane.
A main theorem of this paper is the following.

\begin{theorem}\label{thm:main}
    Let $S \subseteq \F_2^n$ be a Sidon set, let $s = |S|$, and assume $S$ has affine dimension $n$ and $s > 1$.
    \begin{enumerate}
        \item If $\Cay(\gamma_S)$ has exactly $2$ eigenvalues, then $n \in \set{1,2,4}$, and if $n \neq 1$ then $S$ is a $k$-cover with $(n,s) \in \set{(2,3), (4,6)}$.
        \item If $\Cay(\gamma_S)$ has $3$ or more eigenvalues, then $S$ is a $k$-cover if and only if $\Cay(\gamma_S)$ is strongly regular and $S$ is separable.
        \item If $S$ is a $k$-cover and $\Cay(\gamma_S)$ is not complete, then $\Cay(\gamma_S)$ is strongly regular with parameters
    \[
        \parens{
        2^n,
        \binom{s}{2},
        \frac{3k}{2} \cdot  \frac{2^{n+1} + s^2 (s-3)}{s(s-1)},
        \frac{3k}{2} s
        },
    \]
    or equivalently 
    \[
    \parens{
        2^n,
        \binom{s}{2},
        \frac{(s-2)(s^3 - 3s^2 + 2^{n+1})}{4(2^n-s)}
        ,
        \frac{s^2 (s-1)(s-2)}{4(2^n-s)}
        }.
    \]
    \end{enumerate}
\end{theorem}

The rest of this paper is organized as follows.
In \Cref{sec:prelim}, we introduce preliminaries on Sidon sets and provide the necessary cryptographic background.
In \Cref{sec:generalized-crypto}, we generalize known cryptographic results to obtain results about Sidon sets in $\F_2^n$.
For example, we classify $k$-covers as those Sidon sets with minimal linearity.
Then, in \Cref{sec:cayley-graph}, we study the Cayley graph of $\gamma_S$ and prove \Cref{thm:main}.
Moreover, we leave as an open question whether or not the condition on the separability of $S$ is necessary in \Cref{thm:main}.
In \Cref{sec:1cover-dim11}, we will discuss how \Cref{thm:main} gives rise to the unique rank $3$ strongly regular graph with parameters $(2048,276, 44, 36)$ which is typically constructed using the extended binary Golay code.
Finally, in \Cref{sec:new-lower-bounds}, we improve the lower bound on the largest Sidon set in $\F_2^{4t+1}$ for all $t \geq 4$.

\section{Preliminaries}\label{sec:prelim}

\subsection{Background on Sidon sets}
Sidon sets in $\F_2^n$ are defined to be those sets such that no sum of two distinct points appears twice.
Equivalently, Sidon sets in $\F_2^n$ are those sets that such that no four pairwise distinct points sum to $0$.
A very well-known and difficult problem to answer is determining the largest possible size of a Sidon set in $\F_2^n$ (see \cite{czerwinskiPott2023sidon} \cite{gaborThinSidon} \cite{taitwon2021} \cite{czerwinski2024largesidonsets}), and the answer to this question is only known for $n \leq 10$.
By definition, any Sidon set of the largest possible size cannot be extended to a larger Sidon set.
We call any such Sidon set maximal.

\begin{definition}
    If $S \subseteq \F_2^n$ is a Sidon set such that $S = S'$ for every Sidon set $S' \subseteq \F_2^n$ with $S'\supseteq S$, then $S$ is called \textbf{maximal}.
\end{definition}

Maximal Sidon sets can also be characterized by their exclude points.

\begin{definition}
    Let $S \subseteq \F_2^n$.
    The \textbf{excludes} of $S$ is the set 
    \[
    \exc{S} = \set{a+b+c : a,b,c \in S \text{ pairwise distinct}}.
    \]
    Any point in $\exc{S}$ is called an \textbf{exclude point} of $S$.
    Also, if $p \in \F_2^n \setminus S$, then the \textbf{exclude multiplicity} (or \textbf{multiplicity}) $\mult_S(p)$ of $p \in \F_2^n \setminus S$ is defined as 
    \[
        \mult_S(p) = |\set{\set{a,b,c} : a,b,c \in S \text{ pairwise distinct and } a+b+c=p}|.
    \]
\end{definition}

Hence $S \subseteq \F_2^n$ is Sidon if and only if $S$ is disjoint from $\exc{S}$.
Moreover, a Sidon set $S$ is maximal if and only if $\exc{S} = \F_2^n \setminus S$.
Said differently, a Sidon set $S$ is maximal if and only if no points in $\F_2^n \setminus S$ have exclude multiplicity $0$ with respect to $S$.
An interesting class of Sidon sets are those maximal Sidon sets whose exclude points all have the same multiplicity.
These are $k$-covers, and they are those Sidon sets such that every point in $\F_2^n \setminus S$ has the same, nonzero exclude multiplicity.
If $S$ is a $k$-cover, then the value of $k$ is unique and is determined by the equality
\begin{equation}\label{eqn-k-cover-eqn}    
    k = \frac{1}{2^n - |S|} \cdot \binom{|S|}{3},
\end{equation}
and we refer the reader to \cite{quadspaper} for a proof.

\subsection{Cryptographic background}

Much of what is known about Sidon sets in $\F_2^n$ has been discovered by studying almost perfect nonlinear (APN) functions (see, for instance \cite{carlet_apnGraphMaximal}).

\begin{definition}
    Let $F \colon \F_2^n \to \F_2^n$ be a function.
    We call $F$ \textbf{almost perfect nonlinear} (APN) if for any $a \in \F_2^n \setminus \set{0}$ and any $b \in \F_2^n$, there are at most $2$ solutions to the equation $F(x+a)+F(x)=b$.
\end{definition}

For a function $F \colon \F_2^n \to \F_2^n$, we denote by $\delta_F(a,b)$ the number of solutions to the equation $F(x+a)+F(x)=b$.
Hence $F$ is APN if and only if $\delta_F(a,b) \leq 2$ for all $a,b \in \F_2^n$ with $a \neq 0$.
Also, we define $\gamma_F \colon (\F_2^n)^2 \to \F_2$ to be the function such that 
\begin{equation}\label{eq:gammaF}
    \gamma_F(a,b) = 
    \begin{cases}
        1 & a \neq 0 \text{ and } \delta_F(a,b) > 0 \\
        0 & \text{otherwise}.
    \end{cases}
\end{equation}
Both $\delta_F$ and $\gamma_F$ give information on the utility of $F$ in a block cipher and its resistance against so-called differential and linear attacks (see \cite{BihamShamir1991} \cite{NybergBook1994} \cite{MatsuiLinearAttack} \cite{chabaud_vaudenay_1995}).

The \textbf{Walsh transform} of a Boolean function $f \colon \F_2^n \to \F_2$ is the function $W_f \colon \F_2^n \to \F_2$ defined by 
\begin{equation*}
    W_f(a) = \sum_{x \in \F_2^n} (-1)^{f(x) + a \cdot x}
\end{equation*}
for all $a \in \F_2^n$.
A Boolean function $f \colon \F_2^n \to \F_2$ is \textbf{bent} if $|W_f(a)| = 2^{n/2}$ for all $a \in \F_2^n$.
Also, for a vectorial Boolean function $F \colon \F_2^n \to \F_2^n$, the \textbf{Walsh transform} $W_F \colon (\F_2^n)^2 \to \Z$ of $F$ is defined by 
\begin{equation*}
    W_F(a,b) = \sum_{x \in \F_2^n} (-1)^{a \cdot x + b \cdot F(x)}
\end{equation*}
for all $(a,b) \in (\F_2^n)^2$.

The reason that APN functions have been so prominent in the study of Sidon sets is because that APN functions are also exactly those functions $F \colon \F_2^n \to \F_2^n$ whose graph $
    \graph F = \set{(x, F(x)) : x \in \F_2^n}$
is a Sidon set. 
APN functions can also be classified by means of their Walsh transform, and we now recall the Chabaud and Vaudenay classification of APN functions.
\begin{lemma}[\textup{\cite{chabaud_vaudenay_1995}}]\label{lem:CV95Fourth-Moment}
    For any $F \colon \F_2^n \to \F_2^n$, we have 
    \[
        \sum_{\substack{a,b \in \F_2^n \\ b \neq 0}} W_F^4(a,b) \geq 2^{3n+1}(2^n-1)
    \]
    with equality if and only if $F$ is APN.
\end{lemma}
In addition to the Walsh transform, we also use the Fourier-Hadamard transform.
The \textbf{Fourier-Hadamard transform} of a function $\varphi \colon \F_2^n \to \Z$ is the function $\widehat{\varphi} \colon \F_2^n \to \Z$ defined by
\[
    \widehat{\varphi}(a) = \sum_{u \in \F_2^n}(-1)^{u \cdot a} \varphi(u).
\]

For any $\varphi \colon \F_2^n \to \Z$, we have \textbf{Parseval's relation}
\begin{equation*}\label{eq:Parseval}
    \sum_{a \in \F_2^n}(\widehat{\varphi})^2(a) = 2^n \sum_{x \in \F_2^n} \varphi^2(x).
\end{equation*}
When $\varphi$ takes values only in $\set{\pm 1}$ Parseval's relation becomes 
\begin{equation}\label{eq:Parseval-pm1}
   \sum_{a \in \F_2^n}(\widehat{\varphi})^2(a) = 2^{2n}.
\end{equation}

Observe that \Cref{lem:CV95Fourth-Moment} is equivalent to stating $\widehat{W_F^4}(0,0) \geq 2^{2n}(3 \cdot 2^{2n} - 2^{n+1})$ with equality if and only if $F$ is APN.
An important subclass of APN functions are almost bent functions.

\begin{definition}
    Let $F \colon \F_2^n \to \F_2^n$ be a function.
    We call $F$ \textbf{almost bent} (AB) if $W_F(a,b) \in \{0, \pm 2^{\frac{n+1}{2}}\}$ for all $(a,b) \in \F_2^n \setminus \set{(0,0)}$.
\end{definition}

Similar to APN functions, AB functions can also be classified by the Walsh transform.
\begin{theorem}[\textup{\cite{chabaud_vaudenay_1995}}]\label{thm:Chabaud_Vaudenay_WalshBound}
    For any $F \colon \F_2^n \to \F_2^n$, we have 
    \begin{equation}\label{ineq:WalshF-Bound}
        \max_{\substack{a,b \in \F_2^n \\ b \neq 0}} |W_F(a,b)| 
        \geq 
        2^{\frac{n+1}{2}}
    \end{equation}
    with equality if and only if $F$ is APN and $W_F(a,b) \in \set{0, \pm 2^{\frac{n+1}{2}}}$, i.e. $F$ is AB.
\end{theorem}

Another useful and well-known characterization of AB functions is the following.

\begin{theorem}\label{thm:AB-vanDamFlaass}
\textup{\cite{vandamflass}}
    Let $F \colon \F_2^n \to \F_2^n$ be a function.
    Then $F$ is AB if and only if the system of equations
    \[
        \begin{cases}
            x + y + z = a \\
            F(x) + F(y) + F(z) = b
        \end{cases}
    \]
    has $2^n -2$ or $3 \cdot 2^n - 2$ solutions $(x,y,z) \in (\F_2^n)^3$ for every $(a,b) \in (\F_2^n)^2$.
    If so, then the system has $2^n -2$ solutions if $b \neq F(a)$ and $3 \cdot 2^n - 2$ solutions otherwise.
\end{theorem}

Equivalently, a vectorial Boolean function $F \colon \F_2^n \to \F_2^n$ is AB if and only if $\graph F$ is a $\frac{2^n-2}{6}$-cover (recall that all $k$-covers are Sidon sets by definition). 

\subsection{Background from finite geometry}

Now, we recall the following definitions from finite geometry (we refer the reader to \cite{quadspaper} for a detailed explanation of the finite geometry of $\F_2^n$).
\begin{definition}
    Let $S = \set{x_1, \dots, x_k} \subseteq \F_2^n$.
    \begin{enumerate}
        \item An \textbf{affine combination} of points in $S$ is a linear combination $\sum_{i=1}^m \alpha_i x_i$ such that $\alpha_i \in \F_2$ and $\sum_{i=1}^m \alpha_i = 1$.
        \item The \textbf{affine span} $\aff(S)$ of $S$ is the set of all affine combinations of points in $S$.
        \item If $x_i \notin \aff(S \setminus \set{x_i})$ for all $1 \leq i \leq k$, then we say that $S$ is \textbf{affinely independent}.
        \item If $d+1$ is the largest integer such that $S$ contains an affinely independent subset of size $d+1$, then we say that $S$ has \textbf{affine dimension} $d$.
        \item We call a function $A \colon \F_2^n \to \F_2^m$ \textbf{affine} if $A = L + v$ where $L \colon \F_2^n \to \F_2^m$ is a linear map and $v \in \F_2^m$.
        \item We say that $S$ is \textbf{affinely equivalent} to $T \subseteq \F_2^m$, denoted by $S \cong T$, if there exists an affine permutation $A \colon \aff(S) \to \aff(T)$ such that $A(S) = T$. 
    \end{enumerate}
\end{definition}

\section{Generalizing cryptographic results}\label{sec:generalized-crypto}

For any $A,B \subseteq \F_2^n$, we denote by $A+B$ the set $\set{a+b : (a,b) \in A \times B}$.
Given any set $S \subseteq \F_2^n$, we define the function $\delta_S \colon \F_2^n \to \F_2$ by $\delta_S(a) = |(a+S) \cap S|$ for any $a \in \F_2^n$.
Also, define $\gamma_S(a)$ such that 
\[
    \gamma_S(a) = 
    \begin{cases}
        0 & \text{if } a = 0 \text{ or } (a+S) \cap S = \emptyset \\
        1 & \text{otherwise}.
    \end{cases}
\]
for all $a \in \F_2^n$.
Clearly $\gamma_S(a) = 1$ if and only if $a \neq 0$ and $\delta_S(a) > 0$.
The $\gamma_S$ and $\delta_S$ functions are generalizations of the $\gamma_F$ and $\delta_F$ functions when $F \colon \F_2^n \to \F_2^n$ is a vectorial Boolean function.
Moreover, for any $F \colon\F_2^n \to \F_2^n$, we have the equalities $\delta_{\graph F} = \delta_F$ and $\gamma_{\graph F} = \gamma_F$.
APN functions can be classified by properties of $\gamma_F$ and $\delta_F$, and we will prove analogous results for $\gamma_S$ and $\delta_S$.
However, first let us define the function $\Delta_0 \colon \F_2^n \to \F_2$ to be the function such that $\Delta_0$ takes value $1$ only at $0$ and takes value $0$ at all nonzero points. 
We then have the following equivalences.

\begin{proposition}\label{prop:Sidon-iff-Gamma-Delta-Functions}
    Let $S \subseteq \F_2^n$.
    The following are equivalent:
    \begin{enumerate}
        \item $S$ is Sidon
        \item $\wt(\gamma_S) = \binom{|S|}{2}$,
        \item $\delta_S(a) \in \set{0,2}$ for all $a \in \F_2^n \setminus \set{0}$,
        \item $\delta_S = 2\gamma_S + |S|\Delta_0$.
    \end{enumerate}
\end{proposition}
\begin{proof}    
    Notice that 
    \begin{align*}
        \wt(\gamma_S) &= |\set{x \in  \F_2^n \setminus \set{0} :  (x + S) \cap S \neq \emptyset}| \\
        &= |\set{x \in  \F_2^n \setminus \set{0} :  x+s_1 + s_2 = 0 \text{ for some distinct }  s_1, s_2 \in S}| \\
        &= |\set{s_1 + s_2 : s_1, s_2 \in S, s_1 \neq s_2}|.
    \end{align*}
    By definition $S$ is Sidon if and only if each sum of two distinct points in $S$ is distinct.
    Hence (1) and (2) are equivalent.
    
    Now, observe that $|(0+S) \cap S| = |S|$, so $\delta_S(0) = |S|$.
    It was shown in \cite{gaborThinSidon}[Lemma 5] that $S$ is Sidon if and only if $|(a +S) \cap S| \leq 2$ for all $a \neq 0$.
    We then find that (1) and (3) are equivalence because $\delta_S(a) \neq 1$ for all $a \in \F_2^n$ as if $s = a + s' \in (a + S) \cap S$, then $s'=a+s \in (a + S) \cap S$.
    The equivalence of (3) and (4) is also immediate, and so all statements are equivalent.
\end{proof}

Another primary tool that we will use throughout the remainder of this paper is the Fourier transform of a set $S \subseteq \F_2^n$.
For a set $S \subseteq \F_2^n$, we say that the \textbf{Fourier transform} of $S$ is the Fourier-Hadamard transform of its indicator function. 
In other words, the Fourier transform of $S$ is 
\[
\widehat{1_S}(a) =\sum_{x \in \F_2^n}(-1)^{x \cdot a}1_S(x) = \sum_{x \in S} (-1)^{x \cdot a}
\]
where $1_S \colon \F_2^n \to \F_2$ takes value $1$ exactly on the points of $S$.
We call
\[
\mathcal{L}(S) 
    =\max_{a \in \F_2^n \setminus \set{0}} |\widehat{1_S}(a)|
\]
the \textbf{linearity} of $S$.
Similar to before, for any $F \colon \F_2^n \to \F_2^n$, the Fourier transform of $\graph F$ as a subset of $(\F_2^n)^2$ is equal to the Walsh transform of $F$ as a vectorial Boolean function, i.e. $\widehat{1_{\graph F}} = W_F$.
By evaluating the Fourier-Hadamard transform of $(\widehat{1_S})^k$ at point $a \in \F_2^n$, we can count the number of times that $a$ is equal to a sum of $k$ points in $S$.

\begin{proposition}\label{prop:generalized-walsh-to-k}
    Let $S \subseteq \F_2^n$.
    Let $a \in \F_2^n$, and let $k \in \N$.
    Then 
    \[
       \widehat{(\widehat{1_S})^k}(a) 
       = 2^n\left|\set{(x_1, \dots, x_k) \in S^k : x_1 + \cdots + x_k = a}\right|.
    \]
\end{proposition}
\begin{proof}
    Notice that if $(x_1, \dots, x_k) \in S^k$ such that $x_1 + \cdots +x_k = a$, then $\sum_{u \in \F_2^n} (-1)^{u \cdot (x_1 + \cdots + x_k+a)} = 2^n$.
    Moreover, if $(x_1, \dots, x_k) \in S^k$ such that $x_1 + \cdots +x_k \neq a$, then $\sum_{u \in \F_2^n} (-1)^{u \cdot (x_1 + \cdots + x_k+a)} = 0$ as the map $u \cdot u \cdot v$ is balanced for all $v \neq 0$.
    Then
    \begin{align*}
        \widehat{(\widehat{1_S})^k}(a)=
        \sum_{u \in \F_2^n} (-1)^{u \cdot a} (\widehat{1_S})^k(u) 
        &= \sum_{(x_1, \dots, x_k) \in S^k} \sum_{u \in \F_2^n} (-1)^{u \cdot (x_1 + \cdots +x_k+a)} \\
        &= \sum_{\substack{(x_1, \dots, x_k) \in S^k \\ x_1 + \cdots + x_k = a}} \sum_{u \in \F_2^n} (-1)^{u \cdot (x_1 + \cdots +x_k+a)} \\
        &= 2^n\left|\set{(x_1, \dots, x_k) \in S^k : x_1 + \cdots + x_k = a}\right|.
    \end{align*}
\end{proof}

The above proposition is a generalization of the well-known equality 
\[
    \widehat{W_F^k}(a,b) = 2^{2n} \left| \set{ (x_1, \dots, x_k) \in \F_2^n : (x_1, F(x_1)) + \cdots + (x_k, F(x_k)) = (a,b)}\right|
\]
when $F \colon \F_2^n \to \F_2^n$ is any vectorial Boolean function and $(a,b) \in (\F_2^n)^2$.

Note that from \Cref{prop:generalized-walsh-to-k}, we are immediately able to deduce a relation between $\widehat{1_S}$ and $\delta_S$, generalizing the well-known equality $\delta_F = \frac{1}{2^{2n}} \widehat{W_F^2}$ for a vectorial Boolean function $F \colon \F_2^n \to \F_2^n$ (c.f. \cite{chabaud_vaudenay_1995}). 
In particular, for any $S \subseteq \F_2^n$, we have 
\begin{equation}\label{eq:Walsh-Delta-Relation}
    \delta_S = \frac{1}{2^n} \widehat{(\widehat{1_S})^2}.
\end{equation}
Note that \cref{eq:Walsh-Delta-Relation} immediately follows from \Cref{prop:generalized-walsh-to-k} as 
\[
    \widehat{(\widehat{1_S})^2}(a) 
    = 2^n \left|\set{(x,y) \in S^2 : x+y = a}\right|
    =2^n\left| \set{x \in S : x+a \in S}\right| 
    = 2^n\delta_S(a)
\]
for any $a \in \F_2^n$.
Similarly, for any Sidon set $S \subseteq \F_2^n$ and $a \in \F_2^n \setminus S$, we have the equality
\begin{equation}\label{eq:Mult-Walsh-Relation}
        \mult_S(a) 
        = \frac{1}{6 \cdot 2^n} \widehat{(\widehat{1_S})^3}(a)
\end{equation}
because 
\[
\widehat{(\widehat{1_S})^3}(a) = 2^n |\set{(x,y,z) \in S^3 : x+y +z= a}| = 2^n \cdot 3!\cdot\mult_S(a).
\]

Recall the characterization of APN functions from \Cref{thm:Chabaud_Vaudenay_WalshBound}.
Indeed, one can characterize Sidon sets in a similar manner, and this was done in \cite{CarletMesnager2022}.

\begin{lemma}[\cite{CarletMesnager2022}]\label{lem:WalshEquality-IFF-APN}
    Let $S \subseteq \F_2^n$, and let $s = |S|$.
    Then 
    \[
        \widehat{(\widehat{1_S})^4}(0) =\sum_{u \in \F_2^n} (\widehat{1_S})^4(u) \geq 2^n(3s^2 - 2s).
    \]
    with equality if and only if $S$ is Sidon.
\end{lemma}

Moreover, we can generalize \Cref{thm:Chabaud_Vaudenay_WalshBound} to Sidon sets as well.

\begin{theorem}\label{thm:Walsh-Bound-General}
    Suppose $n > 1$, let $S \subseteq \F_2^n$, and let $s = |S|$.
    Then 
    \begin{equation}\label{ineq:LambdaS}
    \mathcal{L}(S) 
    \geq \parens{
    \frac{2^n (3s - 2) - s^3}{2^n - s}
    }^{1/2}
    \end{equation}
    with equality if and only if $S$ is Sidon and $\widehat{1_S}(a) \in \set{0, \pm \mathcal{L}(S)}$ for all $a \in \F_2^n \setminus \set{0}$.
\end{theorem}
\begin{proof}
    As noted in \cite{chabaud_vaudenay_1995}[Proof of Theorem 4], for any function $\varphi \colon \F_2^n \to \Z$, the following holds
    \[
        M = 
        \max_{a \in \F_2^n \setminus \set{0}} \varphi^2(a)
        \geq \frac
        {\sum_{a \in \F_2^n \setminus \set{0}} \varphi^4(a)}
        {\sum_{a \in \F_2^n \setminus \set{0}} \varphi^2(a)}
    \]
    with equality if and only if $\varphi(a) \in \set{0, \pm \sqrt{M}}$ for all $a \in \F_2^n \setminus \set{0}$.
    Note that by \Cref{prop:generalized-walsh-to-k}, we know that 
    \begin{align*}
        \sum_{u \in \F_2^n \setminus \set{0}} (\widehat{1_S})^2(u) 
        &= 2^n \left|\set{(x,y) \in S^2 : x+y=0}\right| - (\widehat{1_S})^2(0) \\
        &= s (2^n - s).
    \end{align*}
    Therefore, by \Cref{lem:WalshEquality-IFF-APN}, we have 
    \begin{align*}
        \mathcal{L}^2(S) = \max_{a \in \F_2^n \setminus \set{0}} W^2_S(a) &\geq \frac{\sum_{a \in \F_2^n \setminus \set{0}} (\widehat{1_S})^4(a)}{\sum_{a \in \F_2^n \setminus \set{0}} (\widehat{1_S})^2(a)}  \\
         &\geq \frac{2^n (3s^2 - 2s) - s^4}{s (2^n - s)} \\
         &= \frac{2^n (3s - 2) - s^3}{2^n - s}.
    \end{align*}
    with equality if and only if $S$ is Sidon and $\widehat{1_S}(a) \in \set{0, \pm \mathcal{L}(S)}$ for all $a \in \F_2^n \setminus \set{0}$, as desired.
\end{proof}

Recall that (\ref{ineq:WalshF-Bound}) is an equality for a function $F \colon \F_2^n \to \F_2^n$ if and only if $F$ is AB.
Using the characterization of AB functions in \Cref{thm:AB-vanDamFlaass}, it follows that that (\ref{ineq:WalshF-Bound}) is an equality if and only if $\graph F$ is a $k$-cover.
In a similar manner, we can generalize this to the context of $k$-covers.
To prove our generalization of \Cref{thm:AB-vanDamFlaass}, we use the same proof technique employed by Carlet to prove \Cref{thm:AB-vanDamFlaass} in \cite{CarletBook}.
We now prove that $k$-covers are exactly the sets in $\F_2^n$ with minimal linearity.

\begin{theorem}\label{thm:k-cover-iff-walsh-characterization}
    Let $S \subseteq \F_2^n$ be a Sidon set.
    Then $S$ is a $k$-cover if and only if ${\widehat{1_S}(a) \in \set{0, \pm \mathcal{L}(S)}}$ for all $a \in \F_2^n \setminus \set{0}$.
\end{theorem}
\begin{proof}
    Let $s = |S|$.
    Observe that $(\widehat{1_S})^3(a) = \mathcal{L}^2(S) \widehat{1_S}(a)$ for all $a \in \F_2^n \setminus \set{0}$ if and only if $\widehat{1_S}(a) \in \set{0, \pm \mathcal{L}(S)}$ for all $a \in \F_2^n \setminus \set{0}$.
    Let $\varphi,\psi \colon \F_2^n \to \Z$ be the functions given by 
    \[
    \varphi(x) = 
    \begin{cases}
            (\widehat{1_S})^3(x) & \text{ if } x \neq 0 \\
            0 & \text{otherwise},
        \end{cases}
    \quad 
    \text{and}
    \quad 
    \psi(x)
        =\begin{cases}
            \mathcal{L}^2(S) \widehat{1_S}(x) & \text{ if } x \neq 0 \\
            0 & \text{otherwise}.
        \end{cases}
    \]
    Then 
    \begin{align*}
        \widehat{\varphi}(a) &= 
        \sum_{u \in \F_2^n \setminus \set{0}} (-1)^{a \cdot u} (\widehat{1_S})^3(u) \\
        &= \sum_{u \in \F_2^n} (-1)^{a \cdot u} (\widehat{1_S})^3(u) - s^3 \\
        &= 2^n \left|\set{(x,y,z) \in S^3 : x +y+z = a}\right| - s^3,
    \end{align*}
    and
    \begin{align*}
        \widehat{\psi}(a) 
        &= \mathcal{L}^2(S) \sum_{u \in \F_2^n \setminus \set{0}}
        (-1)^{a \cdot u} \widehat{1_S}(u)  \\
        &= \mathcal{L}^2(S) \sum_{u \in \F_2^n}
        (-1)^{a \cdot u} \widehat{1_S}(u) - \mathcal{L}^2(S)s  \\
        &= \mathcal{L}^2(S) 2^n \left|\set{x \in S : x= a}\right| - \mathcal{L}^2(S)s.
    \end{align*}
    We know that $\varphi = \psi$ if and only if $\widehat{\varphi} = \widehat{\psi}$ \cite{CarletBook} (note this holds more generally for pseudo-Boolean functions and their Fourier-Hadamard transforms).
    Note that if $a \notin S$, then $\widehat{\varphi}(a) = 6 \cdot 2^n \mult_S(a) - s^3$ and $\widehat{\psi}(a) = -\mathcal{L}^2(S) s$.
    Therefore
    \begin{align*}
        \mult_S(a) &= \frac{1}{6 \cdot 2^n}\parens{s^3 -\mathcal{L}^2(S) s} \\
        &=  \frac{1}{6 \cdot 2^n } \parens{s^3 - s\frac{2^n (3s - 2) - s^3}{2^n - s}} \\
        &= \frac{1}{2^n - s} \cdot \binom{s}{3}
    \end{align*}
    for all $a \in \F_2^n \setminus \set{0}$ if and only if $\widehat{1_S}(a) \in \set{0, \pm \mathcal{L}(S)}$ for all $a \in \F_2^n \setminus \set{0}$.
\end{proof}

Hence $k$-covers can be classified by means of their Fourier transforms, but more specifically, they are the sets in $\F_2^n$ with minimal linearity.
\Cref{thm:k-cover-iff-walsh-characterization} highlights the usefulness in employing cryptographic methods when studying of Sidon sets as we were able to generalize known results on APN and AB functions to obtain classifications of Sidon sets and $k$-covers.

In general, Sidon sets can have linearity close to their size. 
For example, if $S \subseteq \F_2^n$ is a Sidon set of $n+1$ affinely independent points, then $\mathcal{L}(S) = n-1 =|S|-2$ since there is an affine hyperplane containing exactly $n$ points of $S$. 
If we restrict to sets of affine dimension $n$, we have the following bound.

\begin{proposition}\label{prop:Sidon-ZeroNonlinearity}
    Let $S \subseteq \F_2^n$, and let $s = |S|$.
    If $s>1$ and the affine dimension of $S$ is equal to $n$, then $\mathcal{L}(S) \leq s-2$.
\end{proposition}
\begin{proof}
    Assume $S$ has dimension $n$.
    By way of contradiction, suppose that there exists $a \in \F_2^n \setminus \set{0}$ such that $|\widehat{1_S}(a)| =s$.
    Then the map from $S$ to $\F_2$ given by $x \mapsto x \cdot a$ is constant.
    So, there exists $a \in \F_2^n \setminus \set{0}$ and $c \in \F_2$ such that $x \cdot a = c$ for all $x \in S$.
    Then $S$ lies in the affine hyperplane $\set{u \in \F_2^n : u \cdot a = c}$, implying the affine dimension of $S$ must be strictly less than $n$, a contradiction.
    Therefore $\mathcal{L}(S) < s$, and since $\mathcal{L}(S) \equiv s \mod 2$, it follows that $\mathcal{L}(S)\leq s-2$.
\end{proof} 

\begin{remark}
    In the particular case when $S=\graph F$ where $F \colon \F_2^n \to \F_2^n$ is an APN function, the best-known upper bound on $\mathcal{L}(S)$ is $\mathcal{L}(S) \leq 2^n -4$ when $n \geq 3$ \cite{czerwinski2024largesidonsets}.
\end{remark}

A particular class of Sidon sets that will play an important role in the next section are separable Sidon sets.

\begin{definition}
    Let $S \subseteq \F_2^n$.
    We call $S$ \textbf{separable} if there exists an affine hyperplane $H$ of $\F_2^n$ such that $|H \cap S| = \frac{|S|}{2}$.
\end{definition} 

It follows almost immediately that separable sets are those whose Fourier transform contains $0$ in its image.

\begin{proposition}\label{prop:separable-iff-walshZero}
    Let $S \subseteq \F_2^n$.
    Then $S$ is separable if and only if there exists $a \in \F_2^n$ such that $\widehat{1_S}(a)=0$.
\end{proposition}
\begin{proof}
    Suppose $S$ is separable.
    Let $s =|S|$, and let $H$ be an affine hyperplane of $\F_2^n$ such that $|H \cap S| = \frac{s}{2}$.
    It is not difficult to see that there exists $(u,c) \in \F_2^n \times \F_2$ such that $H = \set{x \in \F_2^n : u \cdot x = c}$.
    Since exactly half of the points of $S$ are contained in $H$, the map $x \mapsto x \cdot u$ is balanced as $x$ ranges in value across $S$.
    Therefore $\widehat{1_S}(u) = 0$.
    Conversely, if $\widehat{1_S}(u) = 0$ for some $u \in \F_2^n$, then the map $x \mapsto x \cdot u$ is balanced as $x$ ranges across values in $S$.
    Therefore exactly half of $S$ is contained in the hyperplane $\set{x \in \F_2^n : x \cdot u = 0}$, so $S$ is separable.
\end{proof}

Recall that a function $F \colon \F_2^n \to \F_2^n$ is \textit{plateaued with single amplitude} if ${W_F(u,v) \in\set{0, \pm \lambda}}$ for all $(u,v) \neq (0,0)$.
It is well-known that all APN plateaued functions with single amplitude are AB (c.f. \cite{CarletBook}).
The following corollary is a generalization of this fact.

\begin{corollary}\label{cor:single-amplitude-generalization}
    Let $S \subseteq \F_2^n$ be a Sidon set, and assume $S$ has affine dimension $n$ and that $S$ is separable.
    Then $|\widehat{1_S}|$ has image size $3$ if and only if $S$ is a $k$-cover.
\end{corollary}
\begin{proof}
    Since $S$ has affine dimension $n$, we know that $\widehat{1_S}$ does not take value $s$ on $\F_2^n \setminus \set{0}$ by \Cref{prop:Sidon-ZeroNonlinearity}.
    Moreover, since $S$ is separable, we know that there exists $u \in \F_2^n \setminus \set{0}$ such that $\widehat{1_S}(u) = 0$.
    Therefore, if $|\widehat{1_S}|$ has image size $3$, there exists $\lambda \in \Z$ such that $\widehat{1_S}(a)\in \set{0, \pm \lambda}$ for all $a \in \F_2^n \setminus \set{0}$, implying $S$ is a $k$-cover by \Cref{thm:k-cover-iff-walsh-characterization}.
    Conversely if $S$ is a $k$-cover, then $\widehat{1_S}(a) \in \set{0, \pm \mathcal{L}(S)}$ for all $a \in \F_2^n \setminus \set{0}$ by \Cref{thm:k-cover-iff-walsh-characterization}, implying $|\widehat{1_S}|$ has image size $3$.
\end{proof}

\section{The Cayley graph of \texorpdfstring{$\gamma_S$}{}}\label{sec:cayley-graph}

In \cite{BernasconiCodenotti} and \cite{BernasconiCodenottiVanderkam}, it was shown that the Cayley graph of a Boolean function $f \colon \F_2^n \to \F_2$ is strongly regular with $\lambda = \mu$ if and only if $f$ is bent, that is $\sum_{x \in  \F_2^n}(-1)^{f(x) + u \cdot a} \in \set{\pm 2^{n/2}}$ for all $u \in \F_2^n$.
As proven in \cite{carletCharpinZinovievCodesBentDES}, a function $F \colon \F_2^n \to \F_2^n$ is AB if and only if $\gamma_F$ is bent.
Hence, $\gamma_F$ is bent if and only if $\graph F$ is a $k$-cover by \Cref{thm:AB-vanDamFlaass}.

Strongly regular graphs are particularly interesting because they are exactly those graphs with $3$ eigenvalues. 
Moreover, the eigenvalues of a strongly regular graph determine the $\lambda$ and $\mu$ parameters.
\begin{lemma}[\cite{BrouwerCohenNeumaierDistanceRegularGraphs}]\label{lem:SRG-eigenvalue-relations}
    Suppose $\Gamma$ is a strongly regular graph with parameters $(v,k,\lambda, \mu)$ with $\mu > 0$.
    If $\alpha$ and $\beta$ are the eigenvalues of $\Gamma$ not equal to $k$, then $\lambda = \mu + \alpha + \beta$ and $\mu = k + \alpha\beta$.
\end{lemma}

\subsection{Elementary properties of \texorpdfstring{$\Cay(\gamma_S)$}{}}
Recall that a graph $\Gamma$ is \textbf{regular} or \textbf{$k$-regular} if all vertices of $\Gamma$ have degree $k$.
The Cayley graph of a Boolean function $f \colon \F_2^n \to \F_2$ is always regular with every vertex having degree equal to the \textbf{Hamming weight} of $f$, that is $\Cay(f)$ is $\wt(f)$-regular where $\wt(f) = |\set{x \in \F_2^n : f(x) \neq 0}|$.

The Cayley graph of $\gamma_S$ has particularly interesting properties. 
For example, given any $S \subseteq \F_2^n$, the graph $\Cay(\gamma_S)$ is always \textbf{vertex transitive}, that is, given any two vertices $u,v$ of $\Cay(\gamma_S)$, there exists an automorphism on $\Cay(\gamma_S)$ taking $u$ to $v$ (one such example is the map $x \mapsto x + u + v$).
Moreover, the vertices of $\Cay(\gamma_S)$ corresponding to points in $S$ form a \textbf{clique} (recall that a clique is a set of vertices that are pairwise adjacent).
This is easy to see as if $x,y \in S$, then $x=x+y+y \in (x+y) +S$, so $((x+y) \cap S) \cap S \neq \emptyset$ implying $\gamma_S(x+y)=1$.

Indeed, if two sets $S,T \subseteq \F_2^n$ are affinely equivalent, then $\Cay(\gamma_S)$ and $\Cay(\gamma_T)$ are isomorphic graphs.

\begin{proposition}\label{prop:AffineEquivalence-Implies-Isomorphic}
    Let $S,T \subseteq \F_2^n$.
    If $S \cong T$, then $\Cay(\gamma_S)$ is isomorphic to $\Cay(\gamma_T)$.
\end{proposition}
\begin{proof}
    Suppose $S \cong T$.
    Without loss of generality, assume $\dim(S) =\dim(T)= n$.
    Then there exists an affine isomorphism $A \colon \F_2^n \to \F_2^n$ such that $A(S) = T$.
    We will show that $A$ is a graph isomorphism from $\Cay(\gamma_S)$ to $\Cay(\gamma_T)$.

    Let $u,v \in \F_2^n$ be distinct points and assume $u$ and $v$ are adjacent vertices in $\Cay(\gamma_S)$.
    Then $((u+v) + S)\cap S \neq \emptyset$, and this implies $((A(u) + A(v)) +T) \cap T \neq \emptyset$.
    Moreover, it is also easy to see that if $A(u)$ and $A(v)$ are adjacent, then $u$ and $v$ are adjacent.
    Thus $\varphi$ is a graph isomorphism.
\end{proof}

Moreover, if $T \subseteq S \subseteq \F_2^n$, then $\Cay(\gamma_T)$ is a subgraph of $\Cay(\gamma_S)$.

\begin{proposition}
    Let $T \subseteq S \subseteq \F_2^n$.
    Then $\Cay(\gamma_T)$ is a subgraph of $\Cay(\gamma_S)$.
\end{proposition}
\begin{proof}
    Suppose $u$ and $v$ are adjacent vertices in $\Cay(\gamma_T)$.
    Then $(u+v+T) \cap T \neq \emptyset$.
    Since $T \subseteq S$, it follows that $\emptyset \neq (u+v+T) \cap T \subseteq (u+v+S) \cap S$.
    Thus $u$ and $v$ are adjacent in $\Cay(\gamma_S)$.
\end{proof}

Therefore, if $T \subseteq S \subseteq \F_2^n$ and $\Cay(\gamma_T)$ is connected, then $\Cay(\gamma_S)$ is connected.
In particular, it is known that for a Boolean function $f \colon \F_2^n \to \F_2$ the number of connected components of $\Cay(f)$ is equal $2^{n-d}$ where $d$ is the linear dimension of the support of $f$ \cite{BernasconiCodenotti}.
For $\Cay(\gamma_S)$, we are able to characterize its connectedness more explicitly in terms of containing a particular subgraph.

First, recall that the \textbf{hypercube graph} $Q_n$ is the graph whose vertices are points in $\F_2^n$ and edges connect points whose difference has Hamming weight $1$.
The \textbf{half hypercube graph} $\frac{1}{2}Q_{n-1}$ is then the graph whose vertices are points in $\F_2^n$ with even Hamming weight and edges connect points whose difference has Hamming distance is $2$.
Also, the \textbf{square} of a graph $\Gamma=(V,E)$ is the graph $(V,E')$ where $\set{u,v} \in E'$ if and only if the distance of $u$ and $v$ is at most $2$ in $\Gamma$. 
It is well-known that the square of $Q_n$ is $\frac{1}{2}Q_{n+1}$ (c.f. \cite{GodsilManuscript}).

\begin{proposition}\label{prop:affindep-halfcube}
    Let $S \subseteq \F_2^n$ be a set of $n+1$ affinely independent points.
    Then $\Cay(\gamma_S)$ is isomorphic to $\frac{1}{2}Q_{n+1}$.
\end{proposition}
\begin{proof}
    By \Cref{prop:AffineEquivalence-Implies-Isomorphic} $\Cay(\gamma_S)$ is isomorphic to $\Cay(\gamma_T)$ if $S \cong T$, so we can assume without loss of generality that $S = \set{0, e_1, e_2, \dots, e_n}$.
    Let $x,y \in \F_2^n$ such that $x \neq y$.
    Then $(x + y + S) \cap S \neq \emptyset$ if and only if $x+y = e_i + e_j$ for some $i,j \in [n]$ or $x+y = e_i$ for some $i \in [n]$.
    Therefore $x$ and $y$ are adjacent if and only if $x+y$ has Hamming weight at most $2$.
    Thus $\Cay(\gamma_S)$ is the square of $Q_n$, so $\Cay(\gamma_S)$ is isomorphic to $\frac{1}{2}Q_{n+1}$.
\end{proof}

The following result immediately follows. 
\begin{proposition}\label{prop:Cay(gammaS)-connected}
    Let $S \subseteq \F_2^n$.
    The following are equivalent
    \begin{enumerate}
        \item $\Cay(\gamma_S)$ is connected;
        \item $S$ has affine dimension $n$;
        \item $\Cay(\gamma_S)$ contains $\frac{1}{2}Q_{n+1}$ as a subgraph.
    \end{enumerate}
\end{proposition}

We can also describe the multiplicities of an exclude point $p \in \exc{S}$ of a Sidon set $S$ in terms of the number of edges from $p$ into the clique $S$.

\begin{proposition}\label{prop:multiplicity-intersection-with-translation}
    Let $S \subseteq \F_2^n$ be a Sidon set, and let $p \in \F_2^n \setminus S$.
    In $\Cay(\gamma_S)$, the vertex $p$ is adjacent to exactly $3\mult_S(p)$ vertices in the clique $S$.
\end{proposition}
\begin{proof}
    By definition, the number of vertices in $S$ that are adjacent to $p$ is equal to the size of the set $(p + S) \cap (S + S)$.
    Since $p \notin S$, if $p + x = y+x$ where $x,y,z \in S$, then $x,y,z$ are pairwise distinct.
    Let $\mathcal{T}_S(p)$ be the set 
    \[
    \mathcal{T}_S(p) = \set{\set{x,y,z} \subseteq S : x,y,z \text{ pairwise distinct and } x + y + z = p}.
    \]
    Clearly $\mult_S(p) = |\mathcal{T}_S(p)|$.
    Then 
    \begin{align*}
        (p + S) \cap (S+S) &= \set{s \in \F_2^n :  s = p+x = y+z, \text{ for some $x,y,z \in S$}}\\
        &= \set{s \in \F_2^n :  s = p+x= y+z, \set{x,y,z} \in \mathcal{T}_S(p)} \\ 
         &= \set{p+x, p+y, p+z : \set{x,y,z} \in \mathcal{T}_S(p)}.
    \end{align*}
    Since the elements of $\mathcal{T}_p(S)$ are pairwise disjoint, we then know that the size of $(p + S) \cap (S+S)$ is exactly $3|\mathcal{T}_S(p)| = 3 \mult_S(p)$.
\end{proof}

A \textbf{dominating set} in a graph $\Gamma$ with vertex set $V$ is a set of vertices $D \subseteq V$ such that any $v \in V \setminus D$, there exists $u \in D$ such that $u$ is adjacent to $v$.
By \Cref{prop:multiplicity-intersection-with-translation}, it then follows that a Sidon set $S \subseteq \F_2^n$ is maximal if and only if $S$ is a dominating set in $\Cay(\gamma_S)$.
This is equivalent to the following proposition.
\begin{proposition}
    Let $S \subseteq \F_2^n$ be a Sidon set with $|S|\geq3$.
    Then $S$ is maximal if and only if
    \[
    \prod_{p \in \F_2^n} 
    \sum_{x \in \F_2^n} 1_S(x) \gamma_S(x+p)
    \neq 0.
    \]
\end{proposition}

Moreover, another interesting property is that if $S \subseteq \F_2^n$ is a $k$-cover, then $S$ is also a regular clique in $\Cay(\gamma_S)$.
As defined in \cite{GREAVES2018194}, a $c$\textbf{-regular} clique in a graph $\Gamma$ is a clique $C$ such that there exists a constant $c > 0$ such that any vertex not in $C$ is adjacent to exactly $c$ vertices in $C$.
The following proposition immediately follows from \Cref{prop:multiplicity-intersection-with-translation}.

\begin{proposition}\label{prop:k-cover-IFF-regularclique}
    Let $S \subseteq \F_2^n$.
    Then $S$ is a $k$-cover if and only if $S$ is a $3k$-regular clique in $\Cay(\gamma_S)$.
\end{proposition}

\subsection{Strongly regular graphs and \texorpdfstring{$k$}{}-covers}

In this subsection, we classify $k$-covers in terms of the Cayley graph of $\gamma_S$.
In particular, we will show that $k$-covers give rise to strongly regular graphs.
We do this by studying the eigenvalues of $\Cay(\gamma_S)$ which can be described by the Fourier-Hadamard transform of $\gamma_S$ \cite{BernasconiCodenotti}.

\begin{theorem}[\textup{\cite{BernasconiCodenotti}}]\label{thm:eigenvalues-of-Cayf}
    Let $f \colon \F_2^n \to \F_2$ be a Boolean function.
    Then the eigenvalues of $\Cay(f)$ are exactly the values that $\widehat{f}$ takes.
\end{theorem}

Hence, to compute the eigenvalues of $\Cay(\gamma_S)$, we study the Fourier-Hadamard transform of $\widehat{\gamma_S}$.
For this reason, we express $\widehat{\gamma_S}$ in terms of the Fourier transform of $S$.

\begin{lemma}\label{lem:atleast-2evalues-exactform}
    Let $S \subseteq \F_2^n$.
    Then $\Cay(\gamma_S)$ has at least two eigenvalues, and they are exactly the values that $\widehat{\gamma}_S = \frac{1}{2} ( (\widehat{1_S})^2 - |S|)$ takes. 
\end{lemma}
\begin{proof}
    Clearly $\Cay(\gamma_S)$ has at least two eigenvalues because the only graph with single eigenvalue is $K_1$, but $\Cay(\gamma_S)$ has at least $2$ vertices since $n \geq 1$.
    By \Cref{thm:eigenvalues-of-Cayf}, we know that the eigenvalues of $\Cay(\gamma_S)$ are the values that the Fourier-Hadamard transform of $\gamma_S = \frac{1}{2}(\delta_S - m \Delta_0)$ takes.
    Notice that $\widehat{\gamma}_S = \frac{1}{2} ( (\widehat{1_S})^2 - m)$ as $\widehat{\delta_S} = (\widehat{1_S})^2$ and $\widehat{\Delta_0} = 1$.
\end{proof}

Moreover, in the proof of the main theorem of this section, we will use the following lemma, which generalizes Lemma 4 from \cite{carletCharpinZinovievCodesBentDES}. 

\begin{lemma}\label{lem:Fourier-Hadmard-of-GammaS-chi}
    Let $S \subseteq \F_2^n$ be a Sidon set.
    The Fourier-Hadamard transform of $(\gamma_S)_\chi = (-1)^{\gamma_S}$ is the function $2^n \Delta_0 - (\widehat{1_S})^2 + |S|$.
\end{lemma}
\begin{proof}
    As $S$ is Sidon, we know that $\delta_S = |S|\Delta_0 + 2\gamma_S$.
    Since $\gamma_S$ is a Boolean function, we know that $(\gamma_S)_{\chi} = 1 - 2\gamma_S$, so $(-1)^{\gamma_S} = 1 - \delta_S + |S|\Delta_0$.
    So the Fourier-Hadamard transform of $(\gamma_S)_{\chi}$ is $2^n \Delta_0 - \widehat{\delta_S} + |S|$.
    Recall that $\widehat{(\widehat{f})} = 2^n f$ for any function $f$ on $\F_2^n$.
    Therefore, by \cref{eq:Walsh-Delta-Relation}, we have $\widehat{\delta_S} = (\widehat{1_S})^2$, and we conclude our proof.
\end{proof}

We now classify $k$-covers in graph theoretical terms.

\begin{proof}[Proof of \Cref{thm:main}]
    Assume $\Cay(\gamma_S)$ has exactly $2$ eigenvalues.
    Since $S$ has dimension $n$, we know that $\Cay(\gamma_S)$ is connected by \Cref{prop:Cay(gammaS)-connected}.
    It is well-known that if a connected graph has only $2$ eigenvalues, then it is complete, so $\Cay(\gamma_S)$ is isomorphic to the complete graph on $2^n$ vertices.
    Since $\Cay(\gamma_S)$ is $\binom{s}{2}$-regular by \Cref{prop:Sidon-iff-Gamma-Delta-Functions}, this implies $2^n - 1 = \binom{s}{2}$.
    Moreover, it is well-known that the eigenvalues of the complete graph on $p$ vertices are $p-1$ and $-1$ with multiplicity $1$ and $p-1$, respectively. 
    By \Cref{lem:atleast-2evalues-exactform}, the eigenvalues of $\Cay(\gamma_S)$ are the values that $\frac{1}{2} ( (\widehat{1_S})^2 - s)$ takes.
    So $\frac{1}{2} ( (\widehat{1_S})^2 - s)$ takes value $2^n - 1$ and $-1$ with frequency $1$ and $2^n-1$, respectively.
    Since $\frac{1}{2} ( (\widehat{1_S})^2(0) - s)) = \frac{s^2 - s}{2} 
        = 2^n -1$, it then follows that
    $\frac{1}{2} ( (\widehat{1_S})^2(u) - s) = -1$ for all $u \in \F_2^n \setminus \set{0}$.
    Hence $(\widehat{1_S})^2(u) = s-2$ for all $u \in \F_2^n \setminus \set{0}$.
    By Parseval's relation (\ref{eq:Parseval-pm1}) and \Cref{lem:Fourier-Hadmard-of-GammaS-chi}, we have
    \begin{align*}
        2^{2n} &= \sum_{u \in \F_2^n} \parens{\widehat{(\gamma_S)_\chi}(u)}^2\\
        &= \sum_{u \in \F_2^n} \parens{ 2^n \Delta_0 - (\widehat{1_S})^2(u)  +s }^2 \\
        &= (2^n - s^2 + s)^2 + (2^n -1) (-(s-2) + s)^2 \\
        &=  2^{2n} + 2^{n+1} s + s^2 - 2^{n+1}s^2 -2s^3 + s^4 
        + 4(2^n -1) \\
        &= 2^{2n}-(s^2 -s-2)(-s^2 +s + 2^{n+1} -2) \\
        &=2^{2n}+(s-2)(s+1)(-s^2 +s + 2^{n+1} -2).
    \end{align*}
    Hence $(s-2)(s+1)(-s^2 +s + 2^{n+1} -2)=0$.
    Clearly, if $s=2$, then $n=1$ as $S$ has affine dimension $n$.
    So, suppose $n > 1$.
    Then $s \geq 3$, and so $-s^2 +s + 2^{n+1} -2 = 0$.
    Solving this quadratic yields 
    \[
        s = \frac{\sqrt{2^{n+3} -7}+1}{2},
    \]
    implying $\sqrt{2^{n+3} - 7}$ is an integer, i.e. $2^{n+3} - 7 = x^2$ for some $x \in \Z$.
    The equation $2^{n+3} - 7 = x^2$ is called a Ramanujan-Nagell equation, and  Nagell \cite{Nagell1948} proved that the only solutions to this are 
    \[
        (n+3,x) \in \set{(3,1),(4,3),(5,5),(7,11),(15,181)}.
    \]
    This gives $s \in \set{1,2,3,6,91}$.
    Obviously if $n = 1$, then $S$ is not a $k$-cover because $S$ has no exclude points in this case.
    However, if $n > 1$, then $\Cay(\gamma_S)$ being complete implies $S$ is a $k$-cover by \Cref{prop:k-cover-IFF-regularclique}.
    Also, from (\ref{eqn-k-cover-eqn}), it follows that $s$ is divisible by $3$, so we have $(n,s) \in \set{(2,3), (4,6)}$.
    Thus, the first part of this theorem holds.

    Now, suppose $\Cay(\gamma_S)$ has $3$ or more eigenvalues.
    Suppose $S$ is a $k$-cover.
    Then ${(\widehat{1_S})^2(a) \in \set{0,  \mathcal{L}^2(S)}}$ for all $a \in \F_2^n \setminus \set{0}$ by \Cref{thm:k-cover-iff-walsh-characterization}.
    Note that $\widehat{1_S}$ takes value $0$ by hypothesis, and so $S$ is separable by \Cref{prop:separable-iff-walshZero}.
    Since $\widehat{\gamma}_S$ has the same image size as $(\widehat{1_S})^2$, we know $\Cay(\gamma_S)$ has exactly $3$ eigenvalues with minimum eigenvalue $-\frac{s}{2}$ by \Cref{lem:atleast-2evalues-exactform}.
    Moreover $\Cay(\gamma_S)$ is strongly regular as strongly regular graphs are those that have exactly $3$ eigenvalues (c.f. \cite{SpectraOfGraphsBook}).
    % Converse statement
    Conversely, suppose $\Cay(\gamma_S)$ is strongly regular and $S$ is separable.
    Then $(\widehat{1_S})^2$ takes exactly $3$ values by \Cref{lem:atleast-2evalues-exactform}.
    Then by \Cref{cor:single-amplitude-generalization}, we have that $S$ is a $k$-cover.
    So, the second part of this theorem holds.
     
    Now, assume that $S$ is a $k$-cover and $\Cay(\gamma_S)$ is not complete.
    Since $S$ is a $k$-cover, we have that $(\widehat{1_S})^2(a) \in \set{0, \mathcal{L}^2(S)}$ for all $a \in \F_2^n \setminus \set{0}$ by \Cref{thm:k-cover-iff-walsh-characterization}.
    Hence $\frac{1}{2} ( (\widehat{1_S})^2(a) - s)$ is either $-\frac{s}{2}$ or $\frac{1}{2}(\mathcal{L}^2(S) -s)$ for all nonzero $a \in \F_2^n$.
    Also, note that $\frac{1}{2} ( (\widehat{1_S})^2(0) - s)=\frac{1}{2} ( s^2 - s) = \binom{s}{2}$.    
    So, the eigenvalues of $\Cay(\gamma_S)$ are $-\frac{s}{2}, \frac{1}{2}(\mathcal{L}^2(S)-s)$, and $\binom{s}{2}$.
    Now, we will compute the exact parameters $\lambda$ and $\mu$ of $\Cay(\gamma_S)$.
    Since $S$ is a $k$-cover, it is a maximal Sidon set, implying $S$ induces a dominating set in $\Cay(\gamma_S)$ by \Cref{prop:multiplicity-intersection-with-translation}.
    Hence $\Cay(\gamma_S)$ has diameter $2$, implying $\mu > 0$.
    By \Cref{lem:SRG-eigenvalue-relations}, if $m = \binom{s}{2}$, $\alpha = -\frac{s}{2}$ and $\beta = \frac{1}{2}(\mathcal{L}^2(S)-s)$, then
    \begin{align*}
        \mu &= m + \alpha\beta = \binom{s}{2} - \frac{s}{4}(\mathcal{L}^2(S)-s) = \frac{3k}{2} s , \quad \text{ and } \\
        \lambda &= \mu + \alpha + \beta = \binom{s}{2} - \frac{s}{4}(\mathcal{L}^2(S)-s) - \frac{s}{2} + \frac{1}{2}(\mathcal{L}^2(S)-s) = \frac{3k}{2} \cdot  \frac{2^{n+1} + s^2 (s-3)}{s(s-1)}.
    \end{align*}
    By \cref{eqn-k-cover-eqn}, we have $\frac{3k}{2} = \frac{3}{2}\cdot \frac{1}{2^n - s} \cdot \binom{s}{3} = \frac{s(s-1)(s-2)}{4(2^n-s)}$.
    Hence
    \begin{align*}
        \mu &= \frac{s^2(s-1)(s-2)}{4(2^n-s)}, \quad \text{ and } \\
        \lambda &= \frac{s(s-1)(s-2)}{4(2^n-s)} \cdot \frac{2^{n+1} + s^2 (s-3)}{s(s-1)} 
        = \frac{(s-2)(s^3 - 3s^2 + 2^{n+1})}{4(2^n-s)},
    \end{align*}
    as desired.
\end{proof}

Hence, for $n > 4$, a Sidon set $S \subseteq \F_2^n$ of dimension $n$ is a $k$-cover if and only if $\Cay(\gamma_S)$ is a strongly regular graph and $S$ is separable.
Said differently, for $n > 4$, a Sidon set $S \subseteq \F_2^n$ of dimension $n$ is a $k$-cover if and only if $\Cay(\gamma_S)$ is strongly regular with minimum eigenvalue $-\frac{s}{2}$.

Almost all known $k$-covers are constructed from AB functions, and by \Cref{prop:AffineEquivalence-Implies-Isomorphic}, we know that if $S \subseteq \F_2^n$ is affinely equivalent to the graph of an AB function, then $\gamma_S$ is a bent function \cite{carletCharpinZinovievCodesBentDES}.
However, it is unknown if the converse of this is true as well, so let us classify exactly when $\gamma_S$ is bent.

\begin{theorem}\label{thm:bent-kcover}
    Let $S \subseteq \F_2^n$ be a Sidon set of affine dimension $n$, and let $s = |S|$.
    Then $\gamma_S$ is bent if and only if $s = 2^{n/2}$ and $S$ is a $k$-cover.
    If so, then $S$ is a $(\frac{2^{n/2}-2}{6})$-cover.
\end{theorem}
\begin{proof}
    Let $s = |S|$.
    Suppose $\gamma_S$ is bent.
    Then $W_{\gamma_S}(a) = \sum_{x \in \F_2^n}(-1)^{\gamma_S(x) + x \cdot a} \in \set{\pm 2^{n/2}}$ for all $a \in \F_2^n \setminus \set{0}$.
    Since $W_{\gamma_S} = 2^n \Delta_0 - 2\widehat{\gamma_S}$ (see for instance \cite{CarletBook}[Relation (2.32)]), we have that 
    \[
        W_{\gamma_S}  = 2^n \Delta_0 - (\widehat{1_S})^2 + s.
    \]
    Therefore $(\widehat{1_S})^2(a) \in \set{2^{n/2} + s, 2^{n/2}-s}$ for all $a \in \F_2^n \setminus \set{0}$.
    By \Cref{lem:atleast-2evalues-exactform}, we know that the eigenvalues of $\Cay(\gamma_S)$ are $\set{\binom{s}{2}, 2^{n/2-1}, 2^{n/2-1} - s}$.
    Since $\gamma_S$ is bent, the $\lambda$ and $\mu$ parameters of $\Cay(\gamma_S)$ are equal \cite{BernasconiCodenotti}.
    By \Cref{lem:SRG-eigenvalue-relations}, we have
    \begin{align*}
        \binom{s}{2} + 2^{n/2-1}(2^{n/2-1} -s) &= 
        \mu \\
        &= \lambda  \\
        &= \binom{s}{2} + 2^{n/2-1}(2^{n/2-1} -s)+ 2^{n/2-1}+ 2^{n/2-1} - s,
    \end{align*}
    implying $2^{n/2} -s = 0$.
    Therefore $2^{n/2}  =s$.
    Since $2^{n/2-1} - s = 2^{n/2-1} - 2^{n/2} = -2^{n/2-1} = -\frac{s}{2}$ is an eigenvalue of $\Cay(\gamma_S)$, we know that $S$ is separable.
    Therefore $S$ is a $k$-cover by \Cref{thm:main}.

    Conversely, assume that $s=2^{n/2}$ and $S$ is a $k$-cover.
    We know that $\Cay(\gamma_S)$ does not have exactly $2$ eigenvalues because this would imply $n=4$ by \Cref{thm:main}, but the only $k$-cover in dimension $4$ is of size $6$ (see \cite{quadspaper}), and $6$ is not a  power of $2$.
    Therefore, \Cref{thm:main} then implies $\Cay(\gamma_S)$ is a strongly regular graph with parameters 
    \[
    (v,k,\lambda,\mu)=
    \parens{
    2^n,
        \binom{s}{2},
        \frac{(s-2)(s^3 - 3s^2 + 2^{n+1})}{4(2^n-s)}
        ,
        \frac{s^2 (s-1)(s-2)}{4(2^n-s)}
        }.
    \]
    Observe that 
    \begin{align*}
    \lambda &=
        \frac{(2^{n/2}-2)(2^{3n/2} - 3\cdot2^n + 2^{n+1})}{4(2^n-2^{n/2})} \\
        &= 2^{n/2-2} (2^{n/2}-2),\\
    \mu &= 
        \frac{2^n (2^{n/2}-1)(2^{n/2}-2)}{4(2^n-2^{n/2})} \\
        &= 2^{n/2-2} (2^{n/2}-2),
    \end{align*}
    so $\lambda = \mu$.
    Thus $\gamma_S$ is bent by the main results of \cite{BernasconiCodenotti} and \cite{BernasconiCodenottiVanderkam} as $\Cay(\gamma_S)$ is strongly regular with $\lambda = \mu$.    
\end{proof}

The above theorem, \Cref{thm:bent-kcover}, states that if $S$ is a $k$-cover of size $2^{n/2}$, then $\gamma_S$ is bent.
As previously mentioned, graphs of AB functions satisfy this, but it is unknown if there exist $k$-covers of this same size that are not graphs of AB functions.
Said differently, it is unknown if there exists $n \in \N$ and a $k$-cover in $\F_2^n$ with size $2^{n/2}$ that is not affinely equivalent to the graph of an AB function.
We conjecture that there are no such $k$-covers.

\begin{conjecture}
    Assume $S\subseteq \F_2^n$ is a $k$-cover in $\F_2^n$ with $|S|=2^{n/2}$.
    Then there exists an AB function $F \colon \F_2^{n/2} \to \F_2^{n/2}$ such that $\graph F \cong S$.
\end{conjecture}

We also ask if the additional assumption in \Cref{thm:main} that $S$ be a separable Sidon set is necessary.

\begin{question}\label{question:kcover-iff-SRG}
    Let $n > 4$, and let $S \subseteq \F_2^n$ be a Sidon set of dimension $n$.
    Does it hold true that $S$ is a $k$-cover if and only if $\Cay(\gamma_S)$ is strongly regular?
\end{question}

Note that \Cref{question:kcover-iff-SRG} is the same as asking the following.
\begin{question}\label{question:kcover-iff-WS2ImageSize}
    Let $n > 4$ and let $S \subseteq \F_2^n$ be a Sidon set.
    Does it hold true that $S$ is a $k$-cover if and only if the image set of $(\widehat{1_S})^2$ has size $3$?
\end{question}
If the statement from \Cref{question:kcover-iff-WS2ImageSize} holds true, then $(\widehat{1_S})^2$ having image size $3$ implies that $S$ is separable (i.e. $\widehat{1_S}$ takes value $0$) by \Cref{thm:k-cover-iff-walsh-characterization}.
Therefore, to find a counterexample to \Cref{question:kcover-iff-WS2ImageSize}, one must find a Sidon set $S \subseteq \F_2^n$ such that $|\im (\widehat{1_S})^2| =3$ but $S$ is not separable.

\subsection{\texorpdfstring{The case of $1$-covers in $\F_2^{11}$}{}}\label{sec:1cover-dim11}

As we saw in \Cref{thm:bent-kcover}, bent functions correspond to a particular case of a $k$-cover. 
Since $\Cay(\gamma_S)$ is strongly regular with $\lambda = \mu$ precisely when $\gamma_S$ is bent, it is an area of interest to study the case when $\Cay(\gamma_S)$ is strongly regular but $\lambda \neq \mu$.
In this section, we apply \Cref{thm:main} to obtain such an example.
There do exist other Boolean functions whose Cayley graph satisfies $\lambda \neq \mu$, and a more detailed study of such Boolean functions will appear in \cite{CarletThornburgh}.

It is a consequence of \Cref{thm:main} that any $1$-cover $S$ in $\F_2^{11}$ has a $\gamma_S$ function such that $\Cay(\gamma_S)$ is strongly regular but $\gamma_S$ is not bent.
This is because all $1$-covers in dimension $11$ have size $24$ \cite{DanielSPROJ}.
For example
\begin{align*}
    \{ &0, 1, 2, 4, 8, 16, 32, 64, 128, 231, 256, 318, 512, 760, 851, 909, 1024, 
    1179, 1385, 1492, 1589, \\ &1614, 1954, 2047 \},
\end{align*}
is a $1$-cover in $\F_2^{11}$, note that here we representing points in $\F_2^n$ in their integer form, i.e. $(a_0, \dots, a_{n-1}) \in \F_2^n$ is identified with the integer $\sum_{i=0}^{n-1} a_i 2^i$.
Another such example of a $1$-cover in dimension $11$ is the set
\begin{equation}\label{eqn:1cover-dim11-CarletPicek}
   S= \set{0} \cup \set{x \in \F_{2^{11}}^\ast : x^{23} = 1}
\end{equation}
where $\F_{2^n}^\ast = \F_{2^n} \setminus \set{0}$.
Note that this Sidon set can also be considered as an ellipse along with its ``nucleus'', see \cite{gaborThinSidon}.
It turns out that $11$ is the largest dimension where a $1$-cover exists \cite{DanielSPROJ}, and so both of the examples above of $1$-covers are examples $1$-covers of the maximum size.

Throughout the remainder of this section, we will only use $S$ to denote the $1$-cover from (\ref{eqn:1cover-dim11-CarletPicek}).
Note that we can see that $S$ is Sidon by using Carlet and Picek's construction for sum-free Sidon sets.
Below, a set $S$ in an abelian group $(G, +)$ is called \textbf{sum-free} if $S+S$ contains no points of $S$, i.e. $(S+S) \cap S = \emptyset$.

\begin{theorem}[\textup{\cite{carletPicek}}]\label{thm:CarletPicekSumFreeSidon}
    For all positive integers $n$ and $d$ and $j \in \Z / n\Z$, let $e_j = \gcd(d-2^j, 2^n -1) \in \Z/(2^n -1)\Z$, and let $G_{e_j}$ be the multiplicative subgroup 
    \[
    \{x \in \F_{2^n}^\ast : x^{d-2^j} =1\} 
    = \{x \in \F_{2^n}^\ast : x^{e_j} = 1\}
    \]
    of order $e_j$.
    If the function $F(x) = x^d$ is APN over $\F_{2^n}$, then, for every $j \in \Z/n\Z$, $G_{e_j}$ is a sum-free Sidon set in the additive group of $\F_{2^n}$.
    Moreover, for every $k \neq j$, if $x \in G_{e_k}$, $y \in G_{e_j}$, and $x \neq y^{-1}$, then $(x+1)^{d-2^k} \neq (y+1)^{d-2^j}$.
\end{theorem}

It is a well-known result that the function $F \colon \F_{2^n} \to \F_{2^n}$ defined by $F(x) = x^3$ is an APN function for all $n \in \N$ (see, for instance, \cite{NybergBook1994}).
Observe that if $(d,n,j)=(3,11,8)$, we have that $\gcd(d-2^j, 2^n -1)$ over $\Z/(2^n-1)\Z$ is
\begin{align*}
    \gcd(d-2^j, 2^n -1) = \gcd(3-2^8, 2^11 -1) 
    = \gcd(253, 2047) 
    = 23.
\end{align*}
Hence $S$ is a Sidon set by \Cref{thm:CarletPicekSumFreeSidon}.
Furthermore, we have verified with computer calculations that $S$ is a $1$-cover.

\begin{remark}
    Moreover, we have computed that the automorphism group of $\Cay(\gamma_S)$ is a semidirect product $\Z_2^{11} \rtimes M_{24}$ where $M_{24}$ is the Mathieu group $M_{24}$, one of the sporadic simple groups.
\end{remark}

By \Cref{thm:main}, we know that every $1$-cover in dimension $11$ of size $24$ then gives rise to a strongly regular graph with parameters $(2048,276,44,36)$.
In particular, $\Cay(\gamma_S)$ is the unique rank $3$ strongly regular graph with these parameters, and typically, this strongly regular graph is constructed via the extended binary Golay code (see \cite{BrouwerMaldeghem}).
Note that $\gamma_S$ is not bent because $\lambda \neq \mu$ for this strongly regular graph (this is also an immediate consequence of \Cref{thm:bent-kcover}).

Also, determining the independence number of this graph is an open problem.
The largest size of any independent set found in this strongly regular graph is $72$ \cite{jenrich2023maximalcocliquesstronglyregular}, but it has not been shown that this is the largest possible size.
Utilizing the construction of the rank $3$ strongly regular graph with parameters $(2048, 276, 44, 36)$ via $S$, it may be possible to prove to determine the independence number of this graph.
In short, it would suffice to prove that if $x_1, \dots, x_{73}$ are all distinct points in $\F_{2^{11}}$, then there exists $1 \leq i < j < 73$ such that $x_i + x_j = a + b$ for some $a,b \in S$.
That is, to prove that if given any $73$ distinct points in $\F_{2^{11}}$, there exists at least one pair of distinct points whose sum is a $23$rd root of unity or a sum of two $23$rd roots of unity.
This remains an open problem.
 
\section{New lower bounds on the largest Sidon set in \texorpdfstring{$\F_2^{4t+1}$ for $t \geq 4$}{}}\label{sec:new-lower-bounds}

In pursuit of finding the largest possible size a Sidon sets in $\F_2^n$, there have been numerous constructions of Sidon sets throughout the years (c.f. \cite{carletPicek} \cite{CarletMesnager2022} \cite{czerwinski2024largesidonsets}).
Despite many different constructions, finding the exact maximum size of a Sidon set in $\F_2^n$ is still unknown for $n \geq 10$.
In a recent paper \cite{czerwinski2024largesidonsets}, Czerwinski and Pott used the linearity of a Sidon set to improve the best known constructions.
In particular, they proved the following which we will consider a lemma.

\begin{lemma}[\cite{czerwinski2024largesidonsets}]\label{lem:linearity-construction}
    Let $S \subseteq \F_2^n$ be a Sidon set.
    Then there is a Sidon set in $\F_2^{n-1}$ of size $\frac{|S|+\mathcal{L}(S)}{2}$.
\end{lemma}

In this section, we consider the new lower bound that the authors found for $\F_2^{4t+1}$ with $t \geq 2$, or equivalently $\F_2^{2n-1}$ for odd $n \geq 5$.
In particular, the authors proved the following.

\begin{theorem}[\cite{czerwinski2024largesidonsets}]\label{thm:CzerwinskiPottLowerBound}
    For all odd $n\geq 5$, there is a Sidon set $\F_2^{2n-1}$ with size 
    \begin{equation}\label{eq:SidonLowerbound}
    2^{n-1}  + \frac{1}{2} \parens{\lfloor 2^{n/2+1} + 1 \rfloor - (\lfloor 2^{n/2+1}+1\rfloor \mod 4)} 
    = 2^{n-1} + 2\left\lfloor \frac{2^{n/2+1}+1}{4} \right\rfloor.
    \end{equation}
\end{theorem}

In this section, we improve this lower bound of \Cref{thm:CzerwinskiPottLowerBound} by exactly $1$ via a similar technique.
We first use Carlet and Mesnager's construction of Sidon sets as multiplicative subgroups of $\F_{2^n}^\ast$. 
\begin{proposition}[\cite{CarletMesnager2022}]\label{prop:CarletMesnager-construction}
    Let $n$ and $j$ be positive integers, and let $e = \gcd(2^j + 1, 2^n -1)$.
    Then the multiplicative subgroup $G_e = \set{x \in \F_{2^n} : x^{2^j+1}=1} = \set{x \in \F_{2^n} : x^e = 1}$ of order $e$ is a Sidon set.
\end{proposition}
A special case of \Cref{prop:CarletMesnager-construction} is when $n=2m$ and $j=m$, and in this case, $G_e$ has order $e=\gcd(2^m + 1, 2^{2m}-1)=2^m+1$.
In order to evaluate the linearity of $G_e$, we need to find the largest absolute value that its Fourier transform takes on nonzero values.
By \cite{CarletMesnager2022}[Remark 2.4], the Fourier  transform of $G_e$ takes the form
\begin{equation}\label{eq:Ge-Walsh}
    \widehat{1_{G_e}}(a) = \frac{e}{2^n-1} \sum_{x \in \F_{2^n}^\ast} (-1)^{\tr  \parens{ax^{\frac{2^n-1}{e}}}},
\end{equation}
where $\tr \colon \F_{2^n} \to \F_2$ is the absolute trace function $\tr(x)=\sum_{i=0}^{n-1}x^{2^i}$.

For the rest of this section, we only consider the Sidon set $S = G_e = \set{x \in \F_{2^{2n}} : x^{2^n+1}=1}$.
Observe that (\ref{eq:Ge-Walsh}) becomes
\begin{equation}\label{eq:Ge-Walsh-2ncase}
    \widehat{1_S}(a) 
    = \frac{2^n+1}{2^{2n}-1} 
    \sum_{u \in \F_{2^{2n}}^\ast} 
    (-1)^{
    \tr \parens{au^{\frac{2^{2n-1}}{2^n+1}}}} 
    = \frac{1}{2^n - 1} \sum_{u \in \F_{2^{2n}}^\ast} (-1)^{\tr\parens{a u^{2^n-1}}}
\end{equation}
So, to compute the linearity of $S$, it suffices to know the exact weight of the functions $x \mapsto \tr(ax^{2^n-1})$ over $\F_{2^{2n}}$ where $a \in \F_{2^n}$, and this was done in \cite{CharpinGongHyperbent}.

\begin{lemma}[\textup{\cite{CharpinGongHyperbent}[Lemma 3]}]\label{lem:CharpinGongKloosterman}
    Let $a \in \F_{2^n}$.
    Let $f_a\colon \F_{2^{2n}} \to \F_2$ be the Boolean function given by $f_a(x) = \tr(ax^{2^n-1})$.
    Then $\wt(f_a) = (2^n -1)\parens{2^{n-1} - \frac{K_n(a)}{2}}$ where $K_n(a)$ is the Kloosterman sum
    \[
    K_n(a) = \sum_{y \in \F_{2^n}} (-1)^{\tr\parens{y^{-1} + ay}}.
    \]
\end{lemma}

Therefore, the Fourier transform of $S$ has the following form.
\begin{proposition}
    Let $S = \set{x \in \F_{2^{2n}} : x^{2^n+1} = 1}$.
    Then $\widehat{1_S}(a) = K_n(a)+1$ for all $a \in \F_{2^{2n}}$.
\end{proposition}
\begin{proof}
    By \Cref{lem:CharpinGongKloosterman} and the fact that $\sum_{x \in \F_2^n} (-1)^{f(x)} = 2^n - 2\wt(f)$ for any Boolean function $f \colon \F_2^n \to \F_2$, we have that 
    \[
    \sum_{u \in \F_{2^{2n}}^\ast} (-1)^{\tr\parens{a u^{2^n-1}}} 
    = 
    2^{2n} - 2\wt(f_a) - 1
    =
    2^{2n} - 2(2^n -1)\parens{2^{n-1} - \frac{K_n(a)}{2}} -1,
    \]
    where $f_a$ and $K_n(a)$ are defined as in \Cref{lem:CharpinGongKloosterman}.
    Thus (\ref{eq:Ge-Walsh-2ncase}) becomes 
    \begin{align*}
        \widehat{1_S}(a) 
        &= \frac{1}{2^n - 1}\parens{2^{2n} - 2(2^n -1)\parens{2^{n-1} - \frac{K_n(a)}{2}} - 1} \\
        &= K_n(a) +1.
    \end{align*}
\end{proof}

We now prove the main result of this section.
\begin{theorem}
    Let $S = \set{x \in \F_{2^{2n}} : x^{2^n+1} = 1}$.
    Then
    \[
    \mathcal{L}(S)=1 + 4 \left\lfloor \frac{2^{n/2+1}+1}{4} \right\rfloor,
    \]
    and so there exists a Sidon set in $\F_2^{2n-1}$ of size 
    \[
       2^{n-1} + 2 \left\lfloor \frac{2^{n/2+1}+1}{4}   \right\rfloor  +1
    \]
    for all odd $n \geq 1$.
\end{theorem}
\begin{proof}
    It was shown in \cite{LachaudWolfmannGoppa} that Kloosterman sums take exactly the values in the interval $[-2^{n/2}+1, 2^{n/2}+1]$ that are divisible by $4$.
    Hence $\widehat{1_S}$ takes values in the interval $[-2^{n/2}+2, 2^{n/2}+2]$
    that are $\equiv 1 \mod 4$ (ignoring the value that $\widehat{1_S}$ takes at $0$).
    By a similar argument to which the authors used in \cite{czerwinski2024largesidonsets}, we have
    \begin{align*}
    \mathcal{L}(S) 
    &= 1 + \lfloor 2^{n/2+1}+1 \rfloor - (\lfloor 2^{n/2+1}+1 \rfloor \mod 4)\\
    &= 1 + 4 \left\lfloor \frac{2^{n/2+1}+1}{4} \right\rfloor,
    \end{align*}
    and so there exists a Sidon set of size 
    \[
    2^{n-1} + \frac{1}{2} + \frac{1}{2}\parens{1 + 4 \left\lfloor \frac{2^{n/2+1}+1}{4} \right\rfloor}  
    =  2^{n-1} + 2 \left\lfloor \frac{2^{n/2+1}+1}{4}   \right\rfloor  +1
    \]
    in $\F_2^{2n-1}$ for all odd $n \geq 1$ by \Cref{lem:linearity-construction}.
\end{proof}

\section*{Acknowledgments}
The author thanks Claude Carlet and Larry Rolen for helpful feedback and comments.
The author also thanks Gary Greaves for helpful references on regular and maximal cliques.

\printbibliography

\end{document}